\documentclass[a4paper,11pt]{amsart}

\usepackage{amssymb,amsmath,amsfonts,amsthm,mathabx,mathrsfs,pdfsync,
enumerate,bbm,fancyhdr,enumerate,graphicx, url,afterpage,setspace,geometry,mathabx}

\usepackage[dvipsnames]{xcolor}
\usepackage[colorlinks=true, pdfstartview=FitV, linkcolor=blue, citecolor=blue, urlcolor=blue]{hyperref}

\usepackage[english]{babel}
\usepackage[utf8]{inputenc}

\DeclareMathOperator{\Aut}{Aut}
\DeclareMathOperator{\supp}{supp}

\DeclareMathOperator{\id}{id}

\DeclareMathOperator{\C}{\mathbb{C}}

\DeclareMathOperator{\Z}{\mathbb{Z}}

\DeclareMathOperator{\N}{\mathbb{N}}

\DeclareMathOperator{\h}{\mathcal{H}}

\DeclareMathOperator{\B}{\mathcal{B}}

\DeclareMathOperator{\Tr}{Tr}

\newtheorem{thm}{Theorem}

\newtheorem{lem}[thm]{Lemma}
\newtheorem{obs}[thm]{Observation}
\newtheorem{prop}[thm]{Proposition}
\newtheorem{rem}[thm]{Remark}
\newtheorem{cor}[thm]{Corollary}

\theoremstyle{definition}
\newtheorem{defn}[thm]{Definition}

\begin{document}

\title{von Neumann equivalence and group exactness}

\author{Bat-Od Battseren}
\address{RIMS, Kyoto University, 606-8502 Kyoto, Japan}
\curraddr{}
\email{batoddd@gmail.com}
\thanks{}


\dedicatory{}


\begin{abstract}
We will show that group exactness is a von Neumann equivalence invariant. This result generalizes the previously known fact stating that group exactness is stable under measure equivalence and W*-equivalence.
\end{abstract}

\maketitle
\tableofcontents

\section{Introduction}\label{sec intro}
In \cite{IPR19}, von Neumann equivalence (vNE)  was introduced on the class of discrete countable groups. This new equivalence relation is coarser than both Gromov's measure equivalence (ME) and W*-equivalence (W*E). In the same paper, it was proven that the group properties, amenability, Haagerup property (also known as a-T-menability), Kazhdan's property (T), and proper proximality, are stable under vNE. As its continuation, in \cite{Ish21}, it was proven that some other group properties such as (weak) amenability, (weak) Haagerup property, and Haagerup-Kraus' approximation property (AP) are stable under vNE. As group exactness has many connections with these group properties and is known to be preserved under both ME and W*E \cite{Oz07}, it is natural to ask if group exactness is a vNE invariant. 

The notion of group exactness originated from the theory of $C^*$-algebras \cite{KW99}. A discrete countable group $G$ is said to be \textit{exact} if its reduced group $C^*$-algebra $C^*_\lambda (G)$ is exact, or equivalently if its uniform Roe algebra $C^*_u (G)$ is nuclear \cite{Oz00}. Nowadays, many groups are known to be exact. To name some: all discrete countable linear groups over a field \cite{GHW05}, Gromov's hyperbolic groups \cite{Ada94}, relatively hyperbolic groups with respect to a collection of exact subgroups \cite{Oz06}, groups acting properly and cocompactly on a finite dimensional CAT(0) cubical complex \cite{CN05}, and groups with one relation \cite{Gue02}. To mention one of the most important applications of group exactness as a motivation, a discrete countable exact group $G$ is  coarsely embeddable into a Hilbert space and consequently satisfies the Novikov conjecture \cite{GK02} and the coarse Baum-Connes conjecture \cite{Hig00}.

Recall that two discrete groups $\Gamma$ and $\Lambda$ are measure equivalent (written $\Gamma \sim_{ME}\Lambda$) if there exists a standard measure space $(X,\mu)$ and a measure preserving action $\Gamma\times \Lambda \curvearrowright (X,\mu)$ such that each of the $\Gamma$ and $\Lambda$ actions is free and admits a finite measure fundamental domain. It was introduced in \cite{Gro93} as a measure theoretic analog of quasi-isometry (QI), another equivalence relation on groups that has more geometric spirit. It is worth mentioning that group exactness is also preserved by QI. Recall that two groups $\Gamma$ and $\Lambda$ are W*-equivalent if their group von Neumann algebras $L(\Gamma)$ and $L(\Lambda)$ are *-isomorphic. During the past years, many common invariant properties of ME and W*E are revealed, yet these two equivalence relations are distinct: The finite groups $\Z/2\Z$ and $\Z/3\Z$ are ME but not W*E. Whether ME is coarser than W*E is still open, but with a slight change in the definition of ME, one gets an equivalence relation coarser than both ME and W*E, that is vNE.

Our main result is the following theorem.
\begin{thm}\label{thm}
Let $\Lambda$ be a discrete countable group, and let $\Gamma$ be a $vNE$-subgroup of $\Lambda$. If $\Lambda$ is exact, then so is $\Gamma$. In particular, group exactness is stable under vNE.
\end{thm}
\begin{cor}\label{cor}
Group exactness is stable under ME and W*E.
\end{cor}

In Section \ref{sec prel}, we will recall some definitions and results from literature. In Section \ref{section proof}, we will provide a proof of Theorem \ref{thm}.

\section{Preliminaries}\label{sec prel}

\subsection{Group algebras}
Let $G$ be a discrete countable group. For $1\leq p < \infty$, denote by $\ell^p(G)$ the Banach space of $p$-power summable functions on $G$. Denote by $\ell^\infty (G)$ the Banach space of bounded functions on $G$. The \textit{left regular representation} $\lambda: G\rightarrow \mathcal{U} (\ell^2 (G))$ is defined as $\lambda_x f(y) = f(x^{-1}y)$ for all $x,y\in G$ and $f\in \ell^2 (G)$. The \textit{right regular representation} $\rho: G\rightarrow \mathcal{U} (\ell^2 (G))$ is defined as $\rho_x f(y) = f(yx)$ for all $x,y\in G$ and $f\in \ell^2 (G)$. The $C^*$-algebras generated by $\lambda(G)$ and $\rho(G)$ are called the \textit{left and right reduced group $C^*$-algebras} and denoted by $C^*_\lambda (G)$ and $C^*_\rho (G)$, respectively. The von Neumann algebras generated by $\lambda(G)$ and $\rho(G)$ are called the \textit{group von Neumann algebras} and denoted by $L(G)$ and $R(G)$, respectively. The map $\lambda_x\mapsto \rho_x$ extends to the *-isomorphisms  $C^*_\lambda (G)\cong C^*_\rho (G)$ and $L(G)\cong R(G)$. We also have the isometric *-homomorphism $\theta: \ell^\infty (G)\rightarrow \B (\ell^2(G))$ coming from pointwise multiplication. The $C^*$-algebra generated by $\lambda(G)\cup \theta (\ell^\infty(G))$ is called the  \textit{uniform Roe algebra} and denoted by $C^*_u (G)$.
\subsection{Semifinite von Neumann algebras}\label{subsec standard rep}
Let $M$ be a von Neumann algebra. A map  $\Tr: M_+ \rightarrow[ 0,+\infty ]$  is called a \textit{trace} if it satisfies 
\begin{enumerate}
\item $\Tr(x+y) = \Tr(x) + \Tr (y)$ for all $x,y\in M_+$,
\item $\Tr(kx) = k\Tr(x)$ for all $k \geq 0$ and $x\in M_+$ (conventionally $0\cdot (+\infty)=0$),
\item $\Tr (x^*x) = \Tr(xx^*)$ for all $x\in M$.
\end{enumerate}
A trace $\Tr$ on a von Neumann algebra $M$ is said to be
\begin{enumerate}
\item[•] \textit{faithful} if $\Tr(x) = 0$ implies $x=0$,
\item[•] \textit{normal} if $\sup_i \Tr(x_i) = \Tr(\sup_i(x_i)) $ for any bounded increasing net $(x_i)$ in $M_+$,
\item[•] \textit{semifinite} if for any $x\in M_+$ there exists $y\in M_+$ such that $y\leq x$ and $\Tr(y)< \infty$,
\item[•] \textit{finite} if $\Tr(1)<\infty$.
\end{enumerate} A von Neumann algebra is called \textit{semifinite} if there exists a faithful normal semifinite  trace on it. In the sequel, the couple $(M,\Tr)$  is always a semifinite von Neumann algebra endowed with a faithful normal semifinite  trace. The set
\begin{align*}
\mathfrak{n}_{\Tr} = \{ x \in M :\Tr(x^*x)<\infty \}
\end{align*} happens to be a two sided ideal of $M$ on which the trace extends linearly. Moreover, the sesquilinear form $(x,y)\in \mathfrak{n}_{\Tr} \mapsto \Tr(y^*x)
$ gives a faithful inner product on $\mathfrak{n}_{\Tr}$. Thus its closure with respect to the inner product gives a Hilbert space which we denote by $L^2(M,\Tr)$. There is a natural  faithful normal *-representation $M\rightarrow \B (L^2(M,\Tr))$ given by $x.a = xa$ for all $x\in M$ and $a\in \mathfrak{n}_{\Tr}$. This representation is called the \textit{standard representation} of $(M,\Tr)$.

As an example, suppose that $(M,\Tr)$ is  an abelian von Neumann algebra on a separable Hilbert space. In other words, $(M,\Tr) = (L^\infty (X,\mu),\int d\mu)$ for some standard measure space $(X,\mu)$. Then the standard representation is given by pointwise multiplication of $L^\infty (X,\mu)$ on the Hilbert space $L^2 (X,\mu)$ of square integrable functions .

Another example comes from the group von Neumann algebras. Suppose that $G$ is a discrete countable group and $M=L (G)$ is the group von Neumann algebra. Then for any two faithful normal finite  traces $\Tr_1$ and $\Tr_2$, the representations $L^2(M,\Tr_1)$ and $L^2(M,\Tr_1)$ are unitarily equivalent. This statement is true for even general case, namely for countably decomposable von Neumann algebras \cite[Theorem III.2.6.6]{Bla06}. The \textit{canonical trace} of $L(G)$ is defined  by $\tau (T) = \langle T\delta_e,\delta_e\rangle$ for all $T\in L(G)$. This is a faithful normal finite  trace. The map $\lambda_x \in L(G)\mapsto \delta_x\in \ell^2(G)$ extends to a unitary $L^2(L(G),\tau)\rightarrow \ell^2(G)$ that intertwines the standard representation and the left regular representation, so the standard representation of the group von Neumann algebra is nothing but the left regular representation.
\subsection{ME, W*E, and vNE}
\begin{defn}\label{defn ME}
Two discrete countable groups $\Gamma$ and $\Lambda$ are said to be \textit{measure equivalent} (ME) if there exists a standard measure space $(X,\mu)$ and two commuting, measure preserving, free actions $\Gamma\curvearrowright (X,\mu)$ and $\Lambda\curvearrowright (X,\mu)$ with finite measure fundamental domains. 
\end{defn}
Basic examples come from two lattices of a second countable locally compact group. Finite groups are ME to each other. A more complicated example is that countably infinite amenable groups make a single ME class \cite{OW80}.
\begin{defn}\label{defn W*E}
Two discrete countable groups $\Gamma$ and $\Lambda$ are said to be \textit{W*-equivalent} (W*E) if their group von Neumann algebras $L(\Gamma)$ and $L(\Lambda)$ are *-isomorphic.
\end{defn}
For example, any two ICC discrete amenable groups have *-isomorphic group von Neumann algebras called the hyperfinite $II_1$ factor \cite{Con76}.
\begin{defn} Let $G$, $\Lambda$, and $\Gamma$ be discrete countable groups. Let $(M,\Tr)$ be a semifinite von Neumann algebra endowed with a faithful normal semifinite trace. The group of trace preserving *-automorphisms of $M$ is denoted by $\Aut(M,\Tr)$. A group homomorphism $\sigma: G\rightarrow \Aut(M,\Tr)$ is called a \textit{$G$-action on $(M,\Tr)$}. A \textit{fundamental domain} for the action $\sigma$ is a projection $p\in M$ such that $\sum_{x\in G}\sigma_x (p)= 1$, where the sum converges in the strong operator topology. 
We say that $\Gamma$ is a \textit{von Neumann equivalence subgroup (or $vNE$-subgroup)} of $\Lambda$  if there exists an action
$\sigma: \Gamma\times \Lambda\rightarrow \Aut(M,\Tr)$ and fundamental domains $p$ and $q$ for each of $\Lambda$ and $\Gamma$ actions, respectively, such that the trace $\Tr(p)$ is finite. Furthermore, if the trace $\Tr(q)$ is finite, we say $\Lambda$ and $\Gamma$ are \textit{von Neumann equivalent} and write $\Lambda\sim_{vNE} \Gamma$.
\end{defn}
We restate the following two useful observations from \cite[p.3 and p.12]{IPR19}.
\begin{obs}\label{lem EqRels}
ME, as well as W*E, implies vNE.
\end{obs}
Note that now Corollary \ref{cor}  follows directly from Theorem \ref{thm} and Observation \ref{lem EqRels}.
\begin{obs}
Suppose that we have a trace preserving action $\sigma:G\rightarrow \Aut(M,\Tr)$. If $p\in M$ is a fundamental domain for $\sigma$, then the map $\theta_p:\ell^\infty (G)\rightarrow M$, $\theta_p(f) = \sum_{x\in G} f(x)\sigma_{x^{-1}} (p)$ is a faithful normal *-homomorphism.
\end{obs}
\subsection{Completely bounded maps}
Let $A$ and $B$ be $C^*$-algebras.
A linear map $\phi : A\rightarrow B$ is called \textit{completely bounded} if the completely bounded norm (cb-norm in short)
\begin{align*}
\|\phi\|_{cb}= \sup_{n\in\N} \|\phi\otimes \id_n: A\otimes M_n(\C)\rightarrow B\otimes M_n(\C) \|
\end{align*} is finite. A linear map $\phi : A\rightarrow B$ is called \textit{positive} if it sends positive elements to positive elements, and \textit{completely positive} if the linear maps $\phi\otimes \id_n: A\otimes M_n(\C)\rightarrow B\otimes M_n(\C)$ are all positive. When $A$ is unital and $\phi:A\rightarrow B$ is positive, it is automatically continuous with norm $\|\varphi(1)\|$. Consequently, any completely positive map $\phi$ is completely bounded with cb-norm $\|\varphi(1)\|$. 

Let $X$ be a set, and $k:X\times X \rightarrow \C$ be a bounded map (also called a kernel on $X$).
The \textit{Schur multiplier} associated to the kernel $k$ is the map $S_k: \B(\ell^2(X))\rightarrow \B(\ell^2(X))$ defined by $\langle S_k(T)\delta_y,\delta_x\rangle = k(x,y)\langle T\delta_y,\delta_x\rangle $ for all $T\in \B(\ell^2(X))$ and $x,y\in X$. When $S_k$ is completely bounded, we also say the kernel $k$ is \textit{completely bounded}, and when $S_k$ is completely positive, we say the kernel $k$ is \textit{completely positive} or \textit{positive definite}. Of course for some kernels, the Schur multiplier could be not well defined. The following theorem  characterizes (completely) bounded and (completely) positive Schur multipliers. 
\begin{thm}[{\cite[Theorem 5.1]{Pis01}}]\label{prop cb cp char}
Let $X$ be a set, and $k:X\times X \rightarrow \C$ be a bounded kernel.
\begin{enumerate}
\item The Schur multiplier $S_k$ is  completely bounded with $\|S_k\|_{cb}< 1$ if and only if there is a Hilbert space $\h$ and bounded maps $\xi,\eta:X\rightarrow \h$ such that $k(x,y) = \langle \xi(x),\eta(y)\rangle$, $\|\xi(x)\|< 1$, and $\|\eta(y)\|<1$ for all $x,y\in X$.
\item The Schur multiplier $S_k$ is completely positive if and only if there is a Hilbert space $\h$ and a bounded map $\xi :X\rightarrow \h$ such that $k(x,y) = \langle \xi(x),\xi(y)\rangle$  for all $x,y\in X$.
\end{enumerate}
\end{thm} 
Theorem \ref{prop cb cp char} plays an important role to adapt the definition of $C^*$-algebra exactness to discrete groups. We also use this result in the proof of Theorem \ref{thm}.
\subsection{Group exactness}
Similar to amenability, group exactness can be defined in various ways. The most direct definition would be: A discrete countable group is \textit{exact} if its reduced group $C^*$-algebra is exact. In our use, the following one is more convenient.
\begin{defn}\label{defn property A}
A discrete countable group $G$ is \textit{exact} if for every finite subset $E\subseteq G$ and $\epsilon >0$, there exist a finite symmetric subset $F\subseteq G$ and a bounded map $\xi:G\rightarrow \ell^2 (G)_+$ such that 
\begin{enumerate}
\item $\|\xi(x)\|_2=1$ for all $x\in G$,
\item $|1-\langle\xi(x),\xi(y)\rangle|\leq \epsilon$ whenever $x^{-1}y \in E$, 
\item and $\langle \xi(x),\xi(y)\rangle = 0$ whenever  $x^{-1}y \in G\setminus F$.
\end{enumerate}
\end{defn}

\begin{rem}
As explained in \cite[Theorem 4.3.9]{Wil09}, when a discrete countable group $G$ is endowed with its unique (up to coarse equivalence) left-invariant bounded geometry distance, $G$ has Yu's property A if and only if $G$ is exact. Thus the statements in \cite[Theorem 1.2.4]{Wil09} characterize group exactness. Definition 9 occurs while proving the directions Theorem 1.2.4 $(1)\Rightarrow (3)\Rightarrow (7)$. Indeed, the implications Theorem 1.2.4 (3) $\Rightarrow$ Definition \ref{defn property A} $\Rightarrow$  Theorem 1.2.4 (7) can be easily observed. See \cite{BO08} for more characterizations.

The (left $G$-invariant) \textit{tube} of width $S\subseteq G$ is the smallest left $G$-invariant subset $Tube_{l}(S)$ of $G\times G$ containing $\{e\}\times S$. More precisely, 
\begin{align*}
Tube_{l} (S)= \{ (x,y)\in G\times G: x^{-1}y\in S\}.
\end{align*} With this notation, the condition (2) and (3) are equivalent to say that the kernel $(x,y)\in G\times G\mapsto \langle \xi(x),
\xi (y) \rangle$ is uniformly close to 1 on the tube $Tube_l(E)$ and is supported on a tube $Tube_l (F)$ of finite width $F\subseteq G$. One might consider right $G$-invariant tubes by replacing ``$x^{-1}y$" by ``$xy^{-1}$" and produce a similar statement as Definition \ref{defn property A}, but it gives an equivalent statement. To see that, it is enough to consider the map $\check{\xi}:G \rightarrow \ell^2(G)$ defined by $\check{\xi}(x) = \xi(x^{-1})$ for all $x\in G$.
\end{rem}
The following is another way to define group exactness. A remarkable point of this characterization is that it does not necessarily require the kernel to be supported on a tube of finite width and to be positive definite.
\begin{lem}\label{Lem Property A eq def}
A discrete countable group $G$ is exact if and only if there exists a constant $C>0$ such that for every finite subset $I\subseteq G$ and constant $\varepsilon >0$, there exists a kernel $K:G\times G \rightarrow \C$ satisfying
\begin{enumerate}
\item $\|S_K\|_{cb} \leq C$,
\item $|1-K(x,xy)|\leq \varepsilon$ for all $x\in G$ and $y \in I$,
\item $\sum_{y\in G}|K(x,xy)|^2$ converges uniformly for $x\in G$.
\end{enumerate}
\end{lem}
This lemma is inspired by its $C^*$-exactness version which we restate below.
\begin{prop}[{\cite[Corollary 17.15]{Pis03}}]\label{prop char exactness}
A $C^*$-algebra $A\subseteq \B(\h)$ is exact if and only if the inclusion $A\rightarrow \B(\h)$ can be approximated pointwise by a net of finite rank maps $u_i:A\rightarrow \B(\h)$ with $\sup_i \|u_i\|_{cb}<\infty$.
\end{prop}
\begin{proof}[Proof of Lemma \ref{Lem Property A eq def}]
The ``only if" part is done by considering the positive definite kernels of the form $K(x,y) = \langle \xi(x),\xi(y) \rangle$ where $\xi$ is as in Definition \ref{defn property A}. For the ``if" part, we approximate $K$ by kernels that are supported on  tubes of finite width and apply Proposition \ref{prop char exactness} for $A = C^*_\rho (G)\subseteq \B(\ell^2(G))$.

Take any $(I,\varepsilon)$ and fix $K$ as in the statement. For every large enough finite subset $J\subseteq G$ such that $I\subseteq J$ and $\sum_{y\in J^c}|K(x,xy)|^2\leq 1$ for all $x\in G$, we define the kernel $K_{I,\varepsilon,J} = K_J$ on $G$ by truncating $K$ on the tube of width $J$. More precisely, put
\begin{align*}
K_J(x,y) = \left\{\begin{array}{lr}
        K(x,y), & \text{if } x^{-1}y\in J\\
        0, & \text{otherwise}.
        \end{array}\right.
\end{align*} for all $x,y\in G$. Then the Schur multiplier $S_{K_J}: C^*_\rho (G) \rightarrow \B (\ell^2 (G))$ has finite rank. 

Observe that $(K-K_J)(x,\cdot)\in \ell^2(G)$ and $
(K-K_J)(x,y) = \langle (K-K_J)(x,\cdot),\delta_y \rangle$ for all $x,y\in G$. It follows that
\begin{align*}
\|S_{K_J}\|_{cb} &\leq \|S_K\|_{cb} + \|S_K-S_{K_J}\|_{cb} \leq C + \sup_{x\in G} \|(K-K_J)(x,\cdot)\|_2 
\\
&= C + \left( \sum_{y\in J^c}|K(x,xy)|^2\right )^{1/2} \leq C+1.
\end{align*}

Finally, let us check the pointwise approximation condition for the net $(S_{K_{I,\varepsilon,J}})$. Take any $y\in G$. As $I$ grows, we can assume that $y\in I$. In that case, we have
\begin{align*}
 \left \|\left (\rho_y  - S_{K_J} (\rho_y)\right )\xi\right \|_2= \left \|\sum_{x\in G} (1 - K(x,xy)) \xi (x) \delta_{xy}\right \|_2\leq \varepsilon \|\xi\|_2\quad \text{for all} \quad \xi \in \ell^2(G).
\end{align*} This completes the proof by Proposition \ref{prop char exactness}.
\end{proof}

\section{Proof of Theorem \ref{thm}}\label{section proof}
The proof of Theorem \ref{thm} runs until the end of this section. Suppose that we have an action $\sigma : \Gamma\times \Lambda \rightarrow \Aut (M,\Tr)$ that establishes $\Gamma$ as a $vNE$-subgroup of $\Lambda$. Let $p$ and $q$ be fundamental domains for $\Lambda$ and $\Gamma$ actions, respectively. We assume that the trace is normalized so that $\Tr (p)=1$. Assume that the group $\Lambda$ is exact. Take any finite subset $I\subseteq \Gamma$ and constant $\varepsilon >0$. 
Choose a large enough finite subset $E\subseteq \Lambda$ such that 
\begin{align}\label{choice of E}
 \sum_{s\in E^c} \Tr (\sigma_{\beta} (p)\sigma_{s}(p))\leq \varepsilon^2/16 \quad \text{for all} \quad \beta\in I,
\end{align}
 and choose $\epsilon = \dfrac{\varepsilon^2}{16|E|}$. By the exactness of $\Lambda$, we can find a bounded map $\xi : \Lambda\rightarrow \ell^2(\Lambda)_+$ as in Definition \ref{defn property A}. Put $\xi^s (t)=\xi_t (s^{-1}) = \langle \xi(ts),\xi(t)\rangle$ for all $s,t\in \Lambda$. Recall that $\xi^s\equiv 0$ for all $s\in \Lambda\setminus F$ and that $\supp (\xi_t)\subseteq F$ for all $t\in \Lambda$, where $F\subseteq \Lambda$ is a finite symmetric subset. 
The \textit{induced map} $\hat{\xi}:\Gamma\rightarrow L^2(M,\Tr)\otimes_2 \ell^2(\Lambda)$ and the  \textit{induced kernel} $\hat{k}:\Gamma\times \Gamma\rightarrow \C$ are given by
\begin{align}
\begin{aligned}
&\hat{\xi}(\gamma) = \sum_{s\in \Lambda} \sigma_{(\gamma,s)}(p)p\otimes \xi (s)
\\
&\hat{k}(\gamma,\beta) = \langle \hat{\xi}(\gamma),\hat{\xi}(\beta)\rangle
\end{aligned}
\end{align} for all $\gamma,\beta\in\Gamma$. 
(The induction method above is from \cite{Ish21}. It generalizes the induction method introduced in \cite[Lemma 2.1]{Haa16} for the case where $\Lambda$ is a lattice of a locally compact group, and in \cite[p.3]{Oz12} and \cite[Lemma 2.1]{Jol15} for the case where $\Gamma$ and $\Lambda$ are measure equivalent.) Since we have $\|\hat{\xi}\|_2 \leq 1$,the induced kernel is  positive definite and $\|\hat{k}\|_{cb} \leq 1$  by Proposition \ref{prop cb cp char}.

For each finite subset $S\subseteq \Gamma$, denote 
\begin{align*}
q_S = \sum_{\gamma\in S} \sigma_{\gamma} (q)\quad \text{and}\quad p_S = pq_S.
\end{align*}
The \textit{induced map} $\hat{\xi}_S:\Gamma\rightarrow L^2(M,\Tr)\otimes_2 \ell^2(\Lambda)$ and the \textit{induced kernel} $\hat{k}_S:\Gamma\times \Gamma\rightarrow \C$ \textit{associated to $S$} are given by 
\begin{align}
\begin{aligned}
&\hat{\xi}_S(\gamma) =  \sum_{s\in \Lambda} \sigma_{(\gamma,s)}(p_S)p\otimes \xi (s)
\\
&\hat{k}_S(\gamma,\beta) = \langle \hat{\xi}_S(\gamma),\hat{\xi}(\beta)\rangle.
\end{aligned}
\end{align} for all $\gamma, \beta \in \Gamma$
We will prove that for a large enough $S$ the kernel $\hat{k}_S$ satisfies all conditions of Lemma \ref{Lem Property A eq def}. We will need the following lemma.
\begin{lem}\label{Lem conv}
The following statements are true.
\begin{enumerate}
\item $p_S \rightarrow p$ in $L^2(M,\Tr)$ as $S$ goes to $\Gamma$.
\item $\hat{\xi}_S(\gamma) \rightarrow \hat{\xi}(\gamma)$ in $L^2(M,\Tr)\otimes_2 \ell^2(\Lambda)$ uniformly for $\gamma\in \Gamma$ as $S$ goes to $\Gamma$.
\item $\hat{k}_S \rightarrow \hat{k}$ with respect to the cb-norm (hence also uniformly on $\Gamma\times \Gamma$) as $S$ goes to $\Gamma$.
\end{enumerate}
\end{lem}
\begin{proof}
Since $q_S\rightarrow 1$ in SOT, we have $\|p_S-p\|_2 = \|p(q_S-1)\|_2\rightarrow 0.$
For the second statement, since the set $F$ is finite, we have
\begin{align*}
\left \|(\hat{\xi}-\hat{\xi}_S)(\gamma)\right \|_2^2 &= \sum_{t\in\Lambda}\sum_{s\in\Lambda} \Tr(p \sigma_{(\gamma,t)}(p-p_S^*) \sigma_{(\gamma,s)}(p-p_S))\langle {\xi}(s),{\xi}(t)\rangle
\\
&=\sum_{t\in\Lambda}\sum_{s\in F} \Tr(p \sigma_{(\gamma,t)}(p-p_S^*) \sigma_{(\gamma,ts)}(p-p_S))\langle {\xi}(ts),{\xi}(t)\rangle
\\
&=\sum_{s\in F} \Tr(\sigma_\gamma^{-1}(\theta_p(\xi^s))(p-p_S^*) \sigma_{s}(p-p_S))
\\
&\leq \sum_{s\in F} \|\xi^s\|_\infty \|p-p_S^*
\|_2 \|p-p_S\|_2 \rightarrow 0.
\end{align*}
For the third statement, since $(\hat{k}-\hat{k}_S)(\gamma,\beta) = \langle (\hat{\xi}-\hat{\xi}_S)(\gamma),\hat{\xi}(\beta) \rangle$ for all $\gamma,\beta
\in \Gamma$, the completely bounded norm $\|\hat{k}-\hat{k}_S\|_{cb}$ is bounded by $\sup_{\gamma\in\Gamma} \|(\hat{\xi}-\hat{\xi}_S)(\gamma)\|_2$ which goes to zero by the second statement.
\end{proof}
A direct corollary of Lemma \ref{Lem conv} is that the completely bounded norm $\|\hat{k}_S\|_{cb}$ is bounded by $2$ for all large enough $S$. This checks the first condition of Lemma \ref{Lem Property A eq def} for $\hat{k}_S$. Let us check the second condition of Lemma \ref{Lem Property A eq def} for $\hat{k}_S$. Since $\hat{k}$ and $\hat{k}_S$ are uniformly close, it is enough to check that $|1-\hat{k}(\gamma,\gamma\beta)|\leq 3\varepsilon/4$ for all $\gamma\in \Gamma$ and $\beta\in I$. Put $p_\beta = \sum_{u\in E}p\sigma_{(\beta^{-1},u)}(p)$. The choice (\ref{choice of E}) of $E$ implies that
\begin{align}\label{p-pbeta}
\|p-p_\beta\|_2^2 = \Tr((p-p_\beta)(p-p_\beta^*)) =\sum_{u\in E^c}\Tr(\sigma_\beta (p)\sigma_u(p)) \leq \varepsilon^2/16.
\end{align} By the triangle inequality, we have
\begin{align*}
|1-\hat{k}(\gamma,\gamma\beta)| 
\leq \left |1-\Tr(p_\beta)\right | &+ \left |\Tr(p_\beta) - \sum_{t\in \Lambda}\sum_{s\in F} \Tr(\sigma_{(\gamma,t)}^{-1}(p) \sigma_\beta(p_\beta) \sigma_s(p)) \langle \xi(ts),\xi(t)\rangle\right | 
\\
&+ \left |\sum_{t\in \Lambda}\sum_{s\in F} \Tr(\sigma_{(\gamma,t)}^{-1}(p) \sigma_\beta(p_\beta) \sigma_s(p)) \langle \xi(ts),\xi(t)\rangle -\hat{k}(\gamma,\gamma\beta)\right |.
\end{align*}
By (\ref{p-pbeta}), the first  term is bounded as 
\begin{align*}
|1-\Tr(p_\beta)| = |\Tr(p(p-p_\beta))| \leq \|p\|_2\|p-p_\beta\|_2 \leq \varepsilon/4.
\end{align*} The middle term is also bounded by $\varepsilon/4$. Indeed,
\begin{align*}
&\text{Middle term} = \sum_{s\in \Lambda} \sum_{t\in \Lambda}\Tr(\sigma_{(\gamma,t)}^{-1}(p) \sigma_\beta(p_\beta) \sigma_s(p)) \left (1-\langle \xi(ts),\xi(t)\rangle\right) 
\\
&=\sum_{s\in \Lambda} \Tr(\sigma_{\gamma}^{-1}(\theta_p(1-\xi^s)) \sigma_\beta(p_\beta) \sigma_s(p))
\\
&=\sum_{s\in E} \Tr(\sigma_{\gamma}^{-1}(\theta_p(1-\xi^s)) \sigma_\beta(p_\beta) \sigma_s(p))
\\
&\leq \left (\sum_{s\in E}\Tr(\sigma_s(p) \sigma_{\gamma}^{-1}(\theta_p(|1-\xi^s|^2))\sigma_s(p) )\right )^{1/2}\left (\sum_{s\in E}\Tr(\sigma_s(p)\sigma_\beta(p_\beta^* p_\beta) \sigma_s(p))\right )^{1/2}
\\
&\leq \left (\sum_{s\in E} \|1-\xi^s\|_\infty \Tr(\sigma_s(p) )\right )^{1/2}\left (\Tr(p_\beta^* p_\beta)\right )^{1/2}
\\
&\leq \left (\epsilon |E|\right )^{1/2} \leq \varepsilon /4.
\end{align*}
Let us estimate the last term.
\begin{align*}
(\text{Last term})^2&= \left |\left \langle \hat{\xi}(\gamma), \sum_{t\in \Lambda} \sigma_{(\gamma\beta,t)} (p_\beta - p)p\otimes \xi(t)\right \rangle\right |^2
\\
&\leq \left \|\sum_{t\in \Lambda} \sigma_{(\gamma\beta,t)} (p_\beta - p)p\otimes \xi(t)\right \|^2
\\
&=\sum_{s\in F} \Tr \left ( \sigma_{(\gamma\beta,t)}^{-1}(\theta_p (\xi^s)) \sigma_s (p_\beta^*-p) (p_\beta -p) \right )
\\
&=\sum_{s\in F} \Tr \left ( \sigma_{(\gamma\beta,t)}^{-1}(\theta_p (\xi^s)) \sigma_s \left (\sum_{u\in E} \sigma_{(\beta^{-1},u)} (p)p-p\right ) \left (\sum_{v\in E} p\sigma_{p(\beta^{-1},v)} (p)-p\right ) \right )
\\
&= \Tr \left (\left (\sum_{u\in E} \sigma_{(\beta^{-1},u)} (p)p-p\right ) \left (\sum_{v\in E} p\sigma_{(\beta^{-1},v)} (p)-p\right ) \right )
\\
&= \|p_\beta - p\|_2^2 \leq \varepsilon^2/16.
\end{align*}
Therefore, for a large enough $S\subseteq \Gamma$, we have $\|\hat{k}_S\|_{cb}\leq 2$ and $|1-\hat{k}_S (\gamma,\gamma\beta)|\leq \varepsilon$ for all $\gamma\in \Gamma$ and $\beta\in I$. It remains to prove the last condition of Lemma \ref{Lem Property A eq def} for $\hat{k}_S$. To see that, observe that
\begin{align*}
&|\hat{k}_S (\gamma,\gamma\beta)| = \left |\sum_{t\in \Lambda}\sum_{s\in F} \Tr \left (\sigma_{(\gamma,t)}^{-1} (p) \sigma_\beta (p_S)\sigma_s(p)\right ) \langle \xi(ts),\xi(t)\rangle \right |
\\
&=\left |\sum_{t\in \Lambda} \Tr \left (\sigma_{(\gamma,t)}^{-1} (p) \sigma_\beta (p_S)\theta_p(\xi_t)\right ) \right |
\\
&=\left |\sum_{t\in \Lambda} \Tr \left (\left [\sigma_{(\gamma,t)}^{-1} (p) \sigma_\beta (p_S)\theta_p(\xi_t^{1/2})\right ]\left [\theta_p(\xi_t^{1/2})\sigma_{(\gamma,t)}^{-1} (p) )\right ]\right ) \right |
\\
&\leq \left ( \sum_{ t\in \Lambda} \Tr \left (
\sigma_{(\gamma,t)}^{-1} (p)\sigma_{\beta} (p_S)\theta_p(\xi_t)\sigma_{\beta} (p_S^*)
\sigma_{(\gamma,t)}^{-1} (p)
\right )\right)^{1/2}    \left ( \sum_{t\in \Lambda} \Tr \left ( \sigma_{(\gamma,t)}^{-1}(p)\theta_p(\xi_t)
\sigma_{(\gamma,t)}^{-1}(p)
\right )\right)^{1/2}
\\
&\leq \left ( \sum_{ t\in \Lambda} \sum_{s\in F}\Tr \left (
\sigma_{(\gamma,t)}^{-1} (p)\sigma_{\beta} (p_S)\sigma_s(p)\sigma_{\beta} (p_S^*)
\sigma_{(\gamma,t)}^{-1} (p)
\right )\right)^{1/2}    \left ( \sum_{t\in \Lambda}\sum_{s\in F} \Tr \left ( \sigma_{(\gamma,t)}^{-1}(p)\sigma_s(p)
\sigma_{(\gamma,t)}^{-1}(p)
\right )\right)^{1/2}
\\
&\leq \left ( \sum_{ s\in F} \Tr \left (
\sigma_{\beta} (p_S^*p_S)\sigma_s(p)
\right )\right)^{1/2}  |F|^{1/2}.
\end{align*}
It follows that
\begin{align*}
\sum_{\beta\in\Gamma}\sup_{\gamma\in\Gamma}|\hat{k}_S (\gamma,\gamma\beta)|^2 
&\leq {|F|} \sum_{ s\in F}  \sum_{ \beta\in \Gamma} \Tr \left (
\sigma_{\beta} (p_S^*p_S)\sigma_s(p)
\right )
\\
&\leq {|F|}\sum_{ s\in F}  \sum_{ \beta\in \Gamma} \Tr \left (
\sigma_{\beta} (q_S)\sigma_s(p)
\right )
\\
&\leq {|F|} \sum_{ s\in F}  \sum_{\alpha\in S}\sum_{ \beta\in \Gamma} \Tr \left (
\sigma_{\beta\alpha} (q)\sigma_s(p)
\right )
\\
&= |F|^2 |S| < \infty.
\end{align*}
This completes the proof.

\section*{Acknowledgements}
The author is supported by JSPS fellowship program (P21737). The author is sincerely grateful to Narutaka Ozawa for the hospitality at the Research Institute for Mathematical Sciences (RIMS), Kyoto University, and for the fruitful discussions. The author is also grateful to Ignacio Vergara for bringing von Neumann equivalence to the author's attention.

\bibliographystyle{amsalpha}
\bibliography{mybibfile}
\end{document}